\documentclass{amsart}
\usepackage{amssymb}
\usepackage{a4}
\usepackage{color}
\newtheorem{theorem}{Theorem}[section]
\newtheorem{definition}[theorem]{Definition}
\newtheorem{lemma}[theorem]{Lemma}

\def\qedbox{\hbox{$\rlap{$\sqcap$}\sqcup$}}

\def\GL{\operatorname{GL}}

\makeatletter
 \@addtoreset{equation}{section}
\makeatother
\begin{document}
\title[Geometric realizations]
 {Geometric realizations of generalized algebraic curvature operators}
\author[Gilkey, Nik\v cevi\'c, and Westerman]{P. Gilkey, S. Nik\v cevi\'c, and D. Westerman}
\begin{address}{PG and DW: Mathematics Department, University of Oregon,
Eugene Or 97403 USA.}
\end{address}
\email{gilkey@uoregon.edu and dwesterm@uoregon.edu}
\begin{address}
{SN: Mathematical Institute, Sanu,
Knez Mihailova 35, p.p. 367,
11001 Belgrade,
Serbia. {\it E-mail address:stanan@mi.sanu.ac.rs}}\end{address}

\begin{abstract} We study the 8 natural $\GL$ equivariant geometric realization questions for the space of generalized algebraic
curvature tensors. All but one of them is solvable; a non-zero projectively flat Ricci antisymmetric generalized algebraic
curvature is not geometrically realizable by a projectively flat Ricci antisymmetric torsion free connection.
\end{abstract}
\maketitle
\section{Introduction}

The Ricci tensor and the Weyl projective curvature operator have always been important in mathematical physics. 
They play a central role in our analysis. Projective structures are of particular
interest in both mathematics and in mathematical physics. Weyl \cite{W1922} used projective structures to attempt
a unification of gravitation and electro magnetics by constructing a model of space-time geometry combining both structures. His particular
approach failed for physical reasons but his model is still studied; see \cite{F70,H93,PS91,H82}. More recently, as the field is a vast
one, we refer to a few recent references \cite{BGV05,GK05,I04,KW06,MR05,MV87,NV05} to give a flavor of the context in which these concepts
appear.

Let $M$ be a smooth manifold of dimension $m$. We shall assume $m\ge3$ henceforth to avoid complicating the exposition unduly as the
$2$-dimensional setting is a bit different. Let $\nabla$ be a torsion free connection on the tangent bundle $TM$ and let
$$\mathcal{R}^\nabla(x,y):=\nabla_x\nabla_y-\nabla_y\nabla_x-\nabla_{[x,y]}$$
be the curvature operator; this $(3,1)$ tensor satisfies the identities:
\begin{eqnarray}
&&\mathcal{R}^\nabla(x,y)=-\mathcal{R}^\nabla(y,x),\label{eqn-1.a}\\
&&\mathcal{R}^\nabla(x,y)z+\mathcal{R}^\nabla(y,z)x+\mathcal{R}^\nabla(z,x)y=0\,.\label{eqn-1.b}
\end{eqnarray}
The relation of Equation (\ref{eqn-1.b}) is called the {\it first Bianchi identity}. If $P\in M$, we let 
$\mathcal{R}^\nabla_P\in\otimes^3T_P^*M\otimes T_PM=\otimes^2T_P^*M\otimes\operatorname{End}(T_PM)$ be
the restriction of $\mathcal{R}^\nabla$ to $T_PM$.
 
It is convenient to pass to a purely algebraic context. Let $V$ be a vector space of dimension $m$. A tensor
$\mathcal{A}\in \otimes^2V^*\otimes\text{End}(V)$ satisfying the symmetries given in Equations (\ref{eqn-1.a}) and (\ref{eqn-1.b})
is called a {\it generalized algebraic curvature operator} and we let $\frak{A}(V)\subset \otimes^2V^*\otimes\text{End}(V)$ be the
subspace of all such operators. The fundamental question that we shall be examining in this paper is the extent
to which algebraic properties of
$\mathcal{A}$ can be realized geometrically. Since we are working locally, we may assume without loss of generality that $M=V$.

One has the
following result; although this result is well known, we shall give the proof in Section
\ref{sect-2} for the sake of completeness as it is relatively short and as it contains a basic construction that is fundamental to our later
results.

\begin{theorem}\label{thm-1.1} Let $\mathcal{A}\in\frak{A}(V)$. There exists a torsion free connection $\nabla$ on $TV$ so that
$\mathcal{R}_0^\nabla=\mathcal{A}$. 
\end{theorem}

There are other geometric properties it is natural to study and which are invariant under the action of the general linear group $\GL(V)$,
i.e. which do not depend on the choice of a basis for $V$. In either the algebraic or the geometric setting, one defines the {\it Ricci
tensor}
$\rho(\mathcal{A})\in\otimes^2V^*$ by setting
$$\rho(\mathcal{A})(x,y):=\operatorname{Tr}\{z\rightarrow\mathcal{A}(z,x)y\}\,.$$
Decompose $\rho(\mathcal{A})=\rho_s(\mathcal{A})+\rho_a(\mathcal{A})$ where $\rho_s(\mathcal{A})\in
S^2(V^*)$ is a symmetric bilinear form and where
$\rho_a(\mathcal{A})\in\Lambda^2(V^*)$ is an antisymmetric bilinear form by setting
\begin{eqnarray*}
&&\rho_s(\mathcal{A})(x,y):=\textstyle\frac12\{\rho(\mathcal{A})(x,y)+\rho(\mathcal{A})(y,x)\},\\
&&\rho_a(\mathcal{A})(x,y):=\textstyle\frac12\{\rho(\mathcal{A})(x,y)-\rho(\mathcal{A})(y,x)\}\,.
\end{eqnarray*}
\begin{definition}\label{defn-1.2}
\rm Let $\mathcal{A}\in\mathfrak{A}(V)$.\begin{enumerate}
\item $\mathcal{A}$ is Ricci symmetric if and only if  $\rho(\mathcal{A})\in S^2(V^*)$ i.e. $\rho_a(\mathcal{A})=0$.
\item $\mathcal{A}$ is Ricci antisymmetric if and only if $\rho(\mathcal{A})\in\Lambda^2(V^*)$ i.e. $\rho_s(\mathcal{A})=0$.
\item $\mathcal{A}$ is Ricci flat if and only if $\rho(\mathcal{A})=0$.
\end{enumerate}\end{definition}

We say a connecction $\nabla$ is {\it Ricci symmetric} if the associated Ricci tensor is symmetric; such connections are
also often called {\it equiaffine connections}; they play a central role in many
settings -- see, for example, the discussion in \cite{BGNS06, BDS03, M03, MS99, PSS94}.
Although the following result is well known \cite{SS-62}, we present the
proof in Section
\ref{sect-3} since again the proof is short and the constructions involved play a crucial role in our development. If $(\mathcal{O},x)$ is
a system of local coordinates on $M$, let
$$\nabla_{\partial_{x_i}}\partial_{x_j}=\Gamma_{ij}{}^k\partial_{x_k}$$
define the {\it Christoffel symbols} of $\nabla$; we adopt the {\it Einstein convention} and sum over repeated indices. Set
$$\omega_{\mathcal{O}}:=\Gamma_{ij}{}^jdx^i\,.$$

\begin{theorem}\label{thm-1.3}
Let $\nabla$ be a torsion free connection. The following assertions are equivalent:
\begin{enumerate}
\item  One has that $d\omega_{\mathcal{O}}=0$ for
any system of local coordinates $\mathcal{O}$ on $M$.
\item One has that $\operatorname{Tr}(\mathcal{R}^\nabla)=0$, i.e. $\mathcal{R}^\nabla$ is trace free.
\item The connection $\nabla$ is Ricci symmetric.
\item The connection $\nabla$ locally admits a parallel volume form.
\end{enumerate}\end{theorem}

We will establish the following geometric realizability result in Section \ref{sect-4}:

\begin{theorem}\label{thm-1.4}
Let $\mathcal{A}\in\mathfrak{A}(V)$. Then: \begin{enumerate}
\item  If $\mathcal{A}$ is Ricci symmetric, there exists a Ricci symmetric connection $\nabla$ on $TV$ so
$\mathcal{R}_0^\nabla=\mathcal{A}$.
\item If $\mathcal{A}$ is Ricci antisymmetric, there
exists a Ricci antisymmetric connection  $\nabla$  on $TV$ so
$\mathcal{R}_0^\nabla=\mathcal{A}$.
\item If $\mathcal{A}$ is Ricci flat,  there exists a Ricci flat connection  $\nabla$ on $TV$ so
$\mathcal{R}_0^\nabla=\mathcal{A}$.
\end{enumerate}\end{theorem}

The Ricci tensor $\rho$ defines a natural short exact sequence
which is equivariant with respect to the natural action of $\GL(V)$:
\begin{equation}\label{eqn-1.c}
0\rightarrow\ker(\rho)\rightarrow\frak{A}(V)\rightarrow V^*\otimes V^*\rightarrow0\,.
\end{equation}
If $\Theta\in V^*\otimes V^*$, set
\begin{equation}\label{eqn-1.d}
H(\Theta)(x,y)z:=\Theta(x,y)z-\Theta(y,x)z+\Theta(x,z)y-\Theta(y,z)x\,.
\end{equation}
Clearly $H(\Theta)(x,y)=-H(\Theta)(y,x)$. One verifies that the Bianchi identity is satisfied and thus
$H(\Theta)\in\mathfrak{A}(V)$ by computing:
\begin{eqnarray*}
&&H(\Theta)(x,y)z+H(\Theta)(y,z)x+H(\Theta)(z,x)y\\
&=&\Theta(x,y)z-\Theta(y,x)z+\Theta(x,z)y-\Theta(y,z)x\\
&+&\Theta(y,z)x-\Theta(z,y)x+\Theta(y,x)z-\Theta(z,x)y\\
&+&\Theta(z,x)y-\Theta(x,z)y+\Theta(z,y)x-\Theta(x,y)z\\
&=&0\,.
\end{eqnarray*}


 Let $\{e_i\}$ be a basis for $V$. Let $\{e^i\}$ be the corresponding dual basis for $V^*$. Then:
\begin{eqnarray*}
&&\rho(H(\Theta))(y,z)\\
&=&e^i\{\Theta(e_i,y)z-\Theta(y,e_i)z+\Theta(e_i,z)y-\Theta(y,z)e_i\}\\
&=&\Theta(z,y)-\Theta(y,z)+\Theta(y,z)-m\Theta(y,z)\\
&=&\textstyle\frac{1-m}2\{\Theta(z,y)+\Theta(y,z)\}+\frac{1+m}2\{\Theta(z,y)-\Theta(y,z)\}\\
&=&(1-m)\Theta_s(y,z)-(1+m)\Theta_a(y,z)\,.
\end{eqnarray*}
So modulo a suitable renormalization, $H$ splits the short exact sequence of Equation (\ref{eqn-1.c}).
Let $\frak{W}(V):=\ker(\rho)\subset\frak{A}(V)$ be the space of {\it Weyl projective curvature operators.} 
Let $\mathcal{P}(\mathcal{A})$ be the projection of $\mathcal{A}$ on $\frak{W}(V)$;
\begin{equation}\label{eqn-x}
\mathcal{P}(\mathcal{A})=\mathcal{A}+\textstyle\frac1{m-1}H(\rho_s(\mathcal{A}))+\frac1{1+m}H(\rho_a(\mathcal{A}))\,.
\end{equation}
Following \cite{SSV91} one says that
$\mathcal{A}\in\frak{A}(V)$ is {\it projectively flat} if $\mathcal{P}(\mathcal{A})=0$ or, equivalently, if there exists
$\Theta\in V^*\otimes V^*$ so
$\mathcal{A}=H(\Theta)$. One says that a $\nabla$ is projectively flat the associated curvature operator $\mathcal{R}^\nabla_P$ is
projectively flat for all points
$P$ of
$M$. 
Note that:
$$\begin{array}{ll}
  \dim\{S^2(V^*)\}=\textstyle\frac12m(m+1),&
  \dim\{\mathfrak{W}(V)\}=\textstyle\frac{m^2(m^2-4)}{3},\\
\dim\{\Lambda^2(V^*)\}=\textstyle\frac12m(m-1)\,.\vphantom{\vrule height 11pt}
\end{array}$$
One has the following result of Bokan \cite{B90} and Strichartz \cite{S88}; see also related work of Singer and Thorpe \cite{ST}
in the Riemannian setting:

\begin{theorem}\label{thm-1.5}
There is a $\GL(V)$ equivariant decomposition of $\frak{A}(V)$ into irreducible $\GL(V)$ modules
$\frak{A}(V)=\frak{W}(V)\oplus S^2(V^*)\oplus\Lambda^2(V^*)$.
\end{theorem}

We shall omit the proof of the following result as it plays no role in our analysis and is only included for the sake of
completeness; see \cite{SSV91} for further details:

\begin{theorem}\label{thm-1.6}
 Let $\nabla$ and $\bar\nabla$ be torsion free connections. The following conditions
are equivalent and define the notion of {\rm projective equivalence}:
\begin{enumerate}\item $\mathcal{P}(\mathcal{R}^\nabla)=\mathcal{P}(\mathcal{R}^{\bar\nabla})$.
\item There is a $1$-form $\theta$ so $\nabla_xy-\bar\nabla_xy=\theta(x)y+\theta(y)x$.
\item The unparametrized geodesics of $\nabla$ and of $\bar\nabla$ coincide.
\end{enumerate}\end{theorem}

Theorem \ref{thm-1.5} gives rise to additional geometric representability questions. We will establish the following result in Section
\ref{sect-5}:

\begin{theorem}\label{thm-1.7}
 Let $\mathcal{A}\in\frak{A}(V)$.
\begin{enumerate}
\item If $\mathcal{A}$ is projectively flat, then there exists a projectively flat connection $\nabla$ on $TV$ so that
$\mathcal{R}_0^\nabla=\mathcal{A}$.
\smallbreak\item If $\mathcal{A}$ is projectively flat and Ricci symmetric, then there exists a projectively flat and
Ricci symmetric connection $\nabla$ on $TV$ so that $\mathcal{R}_0^\nabla=\mathcal{A}$.
\end{enumerate}
\end{theorem}

The geometrical realization theorems discussed previously are equivariant with respect to the natural action of the general linear group
$\GL(V)$. The decomposition of $\mathfrak{A}(V)$ as a $\GL(V)$ module
has 3 components so there are 8 natural geometric realization questions which are $\GL(V)$ equivariant. Since the flat connection
realizes the 0 curvature operator, there is only one natural $\GL$ equivariant geometric realization question which is not covered by the
forgoing results. It is answered, in the negative, by the following result which we will establish in Section
\ref{sect-6}:

\begin{theorem}\label{thm-1.8} If $\nabla$ is a projectively flat, Ricci antisymmetric, torsion free connection, then $\nabla$
is flat. Thus if $0\ne\mathcal{A}\in\mathfrak{A}(V)$ is projectively flat and Ricci antisymmetric, then $\mathcal{A}$ is not geometrically
realizable by a projectively flat, Ricci antisymmetric, torsion free connection.
\end{theorem}

The geometric representability theorems of this paper can be summarized in the following table; the non-zero components of
$\mathcal{A}$ are indicated by
$\star$.
$$\begin{array}{|c|c|c|r|c|c|c|r|}\noalign{\hrule}
\mathfrak{W}(V)&S^2(V^*)&\Lambda^2(V^*)&&\mathfrak{W}(V)&S^2(V^*)&\Lambda^2(V^*)&\\
\noalign{\hrule}\star&\star&\star&\text{yes}&0&\star&\star&\text{yes}\\
\noalign{\hrule}\star&\star&0&\text{yes}&0&\star&0&\text{yes}\\
\noalign{\hrule}\star&0&\star&\text{yes}&0&0&\star&\text{no}\\
\noalign{\hrule}\star&0&0&\text{yes}&0&0&0&\text{yes}\\\noalign{\hrule}
\end{array}$$

\section{The proof of Theorem \ref{thm-1.1}}\label{sect-2}

Fix a basis $\{e_i\}$ for $V$. If
$\mathcal{A}\in\mathfrak{A}(V)$, expand
$\mathcal{A}(e_i,e_j)e_k=A_{ijk}{}^\ell e_\ell$. Theorem \ref{thm-1.1} will follow from the following observation:

\begin{lemma}\label{lem-2.1}
Let $\Gamma_{uv}{}^l:=\textstyle\frac13(A_{wuv}{}^l+A_{wvu}{}^l)x^w$ be the Christoffel symbols of
a connection $\nabla$. Then $\nabla$ is torsion free and
$\mathcal{R}^\nabla_0(\partial_{x_i},\partial_{x_j})\partial_{x_k}=A_{ijk}{}^l\partial_{x_l}$.
\end{lemma}

\begin{proof} Clearly $\Gamma_{uv}{}^\ell=\Gamma_{vu}{}^\ell$ so $\nabla$ is torsion free. As $\Gamma$ vanishes at the origin,
we may use the curvature symmetries to compute:
\medbreak\qquad
$\mathcal{R}^\nabla_0(\partial_{x_i},\partial_{x_j})\partial_{x_k}
=\textstyle\left\{\partial_{x_i}\Gamma_{jk}{}^l-\partial_{x_j}\Gamma_{ik}{}^l\right\}(0)\partial_{x_l}$
\smallbreak\qquad\quad
$=\textstyle\frac13\left\{A_{ijk}{}^l+A_{ikj}{}^l-A_{jik}{}^l-A_{jki}{}^l\right\}\partial_{x_l}$
\smallbreak\qquad\quad
$=\textstyle\frac13\left\{A_{ijk}{}^l-A_{kij}{}^l+A_{ijk}{}^l-A_{jki}{}^l\right\}\partial_{x_l}$
\smallbreak\qquad\quad
$=\textstyle A_{ijk}{}^l\partial_{x_l}$.\end{proof}

\section{The proof of Theorem \ref{thm-1.3}}\label{sect-3}
\begin{proof}We have by the first Bianchi identity of Equation (\ref{eqn-1.b}) that
$$\operatorname{Tr}\{\mathcal{R}(x,y)\}-\rho(y,x)+\rho(x,y)=0\,.$$
This shows that Assertions (2) and (3) of Theorem \ref{thm-1.3} are equivalent. Note
\begin{eqnarray*}
&&\mathcal{R}_{ijk}{}^l\partial_{x_\ell}=\nabla_{\partial_{x_i}}\nabla_{\partial_{x_j}}\partial_{x_k}
-\nabla_{\partial_{x_j}}\nabla_{\partial_{x_i}}\partial_{x_k}\\
&&\qquad=\{\partial_{x_i}\Gamma_{jk}{}^\ell-\partial_{x_j}\Gamma_{ik}{}^\ell
+\Gamma_{in}{}^\ell\Gamma_{jk}{}^n-\Gamma_{jn}{}^\ell\Gamma_{ik}{}^n\}\partial_{x_\ell},\\
&&\operatorname{Tr}\{\mathcal{R}_{ij}\}dx_i\wedge
dx_j=\{\partial_{x_i}\Gamma_{jk}{}^k-\partial_{x_j}\Gamma_{ik}{}^k
+\Gamma_{in}{}^k\Gamma_{jk}{}^n-\Gamma_{jn}{}^k\Gamma_{ik}{}^n\}dx_i\wedge
dx_j\\
&&\qquad=\textstyle\{\partial_{x_i}\Gamma_{jk}{}^k-\partial_{x_j}\Gamma_{ik}{}^k\}dx_i\wedge dx_j
  =2d\{\Gamma_{ij}{}^jdx_i\}\,.
\end{eqnarray*}
Thus Assertions (1) and (2) of Theorem \ref{thm-1.3} are equivalent. Finally, we compute:
$$\nabla_{\partial_{x_i}}\{e^{\Phi}dx_1\wedge...\wedge dx_m\}
=\textstyle\{\partial_{x_i}\Phi-\sum_k\Gamma_{ik}{}^k\}\{e^\Phi dx_1\wedge...\wedge dx_m\}\,.$$

Thus there exists a parallel volume form on $\mathcal{O}$ $\Leftrightarrow$ $\Gamma_{ik}{}^kdx_i$
is exact. As every closed $1$-form is locally exact, Assertions (1) and (4) of Theorem \ref{thm-1.3} are equivalent.
\end{proof}

\section{The proof of Theorem \ref{thm-1.4}}\label{sect-4}
We extend the discussion in \cite{GN08}. Fix a basis $\{e_1,...,e_m\}$ for $V$ to identify $V$ with $\mathbb{R}^m$; let $x=(x_1,...,x_m)$
be the induced system of coordinates on $V$. Since any neighborhood of $0\in V$ contains an open subset which is real analytically
diffeomorphic to all of $V$, Theorem \ref{thm-1.4} will follow from the following result which is of interest in its own right:
\begin{theorem}\label{thm-4.1}
Let $\mathcal{A}\in\mathfrak{A}(V)$. There exists a torsion free real analytic connection $\nabla$ defined on an open neighborhood
$\mathcal{O}$ of
$0\in V$ so that $\mathcal{R}_0^\nabla=\mathcal{A}$ and so that
$\rho(\mathcal{R}_P^\nabla)(\partial_{x_i},\partial_{x_j})=\rho(\mathcal{A})(e_i,e_j)$ for all $P\in\mathcal{O}$.
\end{theorem}

The remainder of this section is devoted to the proof of Theorem \ref{thm-1.4}. We complexify and set
$V_{\mathbb{C}}:=V\otimes_{\mathbb{R}}\mathbb{C}=\mathbb{C}^m$. Let
$$|z|:=(|z_1|^2+...+|z_m|^2)^{1/2}\quad\text{and}\quad B_\delta:=\{z\in\mathbb{C}^m:|z|<\delta\}$$
be the Euclidean length of $z\in\mathbb{C}^m$ and the open ball of radius $\delta>0$ about the origin, respectively. Let
$\mathcal{H}_\delta$ be the ring of all holomorphic functions $q$ on $B_\delta$ such that $q(x)$ is real for
$x\in\mathbb{R}^m\subset\mathbb{C}^m$. For
$\nu=0,1,2,...$ and for
$q\in\mathcal{H}_\delta$, set
$$||q||_{\delta,\nu}:=\sup_{0<|z|<\delta}|q(z)|\cdot |z|^{-\nu}$$
where, of course, $||q||=\infty$ is possible. Let $\mathcal{H}(\delta,\nu):=\{q\in\mathcal{H}_\delta:||q||_{\delta,\nu}<\infty\}$;
$(\mathcal{H}(\delta,\nu),||_{\delta,\nu})$ is a Banach space. Clearly if $q\in\mathcal{H}(\delta,\nu)$, then we have the estimate
$$|q(z)|\le||q||_{\delta,\nu}\cdot|z|^\nu\quad\text{for all }\quad z\in B_\delta\,.$$
It is immediate that $\oplus_\nu\mathcal{H}(\delta,\nu)$ is a graded ring, i.e.
$$\mathcal{H}(\delta,\nu)\cdot\mathcal{H}(\delta,\mu)\subset\mathcal{H}(\delta,\mu+\nu)\,.$$
If $W$ is an auxiliary real vector space, we let $\mathcal{H}(\delta,\nu,W):=W\otimes_{\mathbb{R}}\mathcal{H}(\delta,\nu)$ be the
appropriate function space of holomorphic functions taking values in $W$; of particular interest will be the function spaces
$\mathcal{H}(\delta,\nu,S^2(V^*))$ and $\mathcal{H}(\delta,\nu,\mathfrak{A}(V))$. Given a basis $\{f_i\}$ for $W$ and given
$q\in\mathcal{H}(\delta,\nu,W)$, we expand $q=\sum_iq_if_i$ for $q_i\in\mathcal{H}(\delta,\nu)$ and define a Banach norm on
$\mathcal{H}(\delta,\nu,W)$, by setting
$$||q||_{\delta,\nu}:=\sup_i||q_i||_{\delta,\nu}\,.$$
Changing the basis for $W$ yields an equivalent norm.

We shall use the canonical coordinate frame to
identify $S^2(T^*V)$ with $V\times S^2(V^*)$ henceforth. The proof of Theorem \ref{thm-4.1} will be based on the following technical Lemma:
\begin{lemma}\label{lem-4.2}
Let $\mathcal{A}\in\mathfrak{A}(\mathbb{R}^m)$. There exists $\delta=\delta(\mathcal{A})>0$,
$C=C(\mathcal{A})>0$, and a sequence $\Gamma_\nu\in\mathcal{H}(\delta,2\nu-1,S^2(V^*)\otimes V)$ for $\nu=1,2,...$ so that:
\begin{enumerate}
\item $\Gamma_{1,uv}{}^l:=\textstyle\frac13(A_{wuv}{}^l+A_{wvu}{}^l)x^w$. 
\item $||\Gamma_\nu||_{\delta,2\nu-1}\le C^{2\nu-1}$.
\item $\Gamma_{\nu,ij}{}^j=0$ for $\nu\ge2$.
\item If $\nabla_\nu$ has
Christoffel symbol $\Gamma_1+...+\Gamma_\nu$, then $||\rho_s(\mathcal{R}^{\nabla_\nu})-\rho_s(\mathcal{A})||_{\delta,2\nu}\le C^{2\nu}$.
\end{enumerate}\end{lemma}

We suppose for the moment such a sequence has been constructed. Choose $\delta_1<\delta$ so $C^2\delta_1<1$.  We set
$\Gamma:=\Gamma_1+\Gamma_2+...$. By Assertion (2), this series converges uniformly for $z\in B_\delta$. Thus the
associated connection $\nabla$ is a real torsion free connection on the real ball of radius $\delta_1$ in $V$. Since uniform convergence in
the holomorphic context implies the uniform convergence on compact subsets of all derivatives, $\Gamma$ is a real analytic
connection near
$0\in V$ with
$R^\nabla=\lim_{\nu\rightarrow\infty}R^{\nabla_\nu}$. Since $\Gamma_\nu=\Gamma_1+O(|x|^3)$, we apply Lemma \ref{lem-2.1} to see
$\mathcal{R}_0^{\nabla_\nu}=\mathcal{R}_0^{\nabla_1}=\mathcal{A}$. 
Define $\mathcal{L}(\Gamma_\nu)$ and $\Gamma_\nu\star\Gamma_\mu$ by setting:
\begin{eqnarray*}
&&\mathcal{L}(\Gamma_\nu)_{ijk}{}^l:=\partial_{z_i}\Gamma_{\nu,jk}{}^l-\partial_{z_j}\Gamma_{\nu,ik}{}^l,\\
&&\{\Gamma_\mu\star\Gamma_\nu\}_{ijk}{}^\ell:=
\Gamma_{\mu,in}{}^\ell\Gamma_{\nu,jk}{}^n
+\Gamma_{\nu,in}{}^\ell\Gamma_{\mu,jk}{}^n
-\Gamma_{\mu,jn}{}^\ell\Gamma_{\nu,ik}{}^n-\Gamma_{\nu,jn}{}^\ell\Gamma_{\mu,ik}{}^n\,.
\end{eqnarray*}
We then have
\begin{eqnarray}
&&\mathcal{R}^{\nabla_\nu}=\sum_{\mu\le\nu}\mathcal{L}(\Gamma_\mu)+\textstyle\frac12\left\{\sum_{1\le\mu_1\le\nu}\Gamma_{\mu_1}\right\}\star
\left\{\sum_{1\le\mu_2\le\nu}\Gamma_{\mu_2}\Gamma_{\mu_2}\right\}\nonumber\\
&&\textstyle\qquad=\mathcal{R}^{\nabla_{\nu-1}}+\mathcal{L}(\Gamma_\nu)+\left\{\sum_{1\le\mu_1\le\nu}\Gamma_{\mu_1}\right\}\star\Gamma_\nu-
\textstyle\frac12\Gamma_\nu\star\Gamma_\nu,\label{eqn-4.a}\\
&&\left\{\rho(\mathcal{L}(\Gamma_\nu))\right\}_{jk}:=
\partial_{x_i}\Gamma_{\nu,jk}{}^i-\partial_{x_j}\Gamma_{\nu,ik}{}^i\,.\nonumber
\end{eqnarray}
It is immediate from the definition that:
\begin{eqnarray*}
\rho(\Gamma_\mu\star\Gamma_\nu)_{jk}&=&\Gamma_{\mu,\ell n}{}^\ell\Gamma_{\nu,jk}{}^n
+\Gamma_{\nu,\ell n}{}^\ell\Gamma_{\mu,jk}{}^n
-\Gamma_{\mu,jn}{}^\ell\Gamma_{\nu,\ell k}{}^n-\Gamma_{\nu,jn}{}^\ell\Gamma_{\mu,\ell k}{}^n\\
&=&\rho(\Gamma_\mu\star\Gamma_\nu)_{kj}.
\end{eqnarray*}
Furthermore, if $\nu\ge2$, then Assertion (3) yields that $\Gamma_{\nu,ik}{}^i=0$ and thus $(\rho(\mathcal{L}(\Gamma_\nu)))$ is symmetric as
well. Consequently
$\rho_a(\mathcal{R}^{\nabla_\nu})=\rho_a(\mathcal{L}(\Gamma_1))=\rho_a(\mathcal{A})$. Thus by Assertion (4), we have
\begin{eqnarray*}
&&|\{\rho(\mathcal{R}^\nabla)(z)-\rho(\mathcal{A})\}_{ij}|\le
\lim_{\nu\rightarrow\infty}||\rho(\mathcal{R}^{\nabla_\nu})-\rho(\mathcal{A})||_{\delta_1,2\nu}\cdot|z|^{2\nu}\\
&\le&\lim_{\nu\rightarrow\infty}||\rho_s(\mathcal{R}^{\nabla_\nu})-\rho_s(\mathcal{A})||_{\delta_1,2\nu}\cdot|z|^{2\nu}=0\,.
\end{eqnarray*}
Thus $\rho(\mathcal{R}^\nabla)=\rho(\mathcal{A})$ as desired and the proof of Theorem \ref{thm-1.4} will be complete once Lemma
\ref{lem-4.2} is established.

Before establishing Lemma \ref{lem-4.2}, we must establish the following solvability result:

\begin{lemma}\label{lem-4.3} If
$\Theta\in\mathcal{H}(\delta,\nu,S^2(V^*))$, there exists $\Gamma\in\mathcal{H}(\delta,\nu+1,S^2(V^*)\otimes V)$ so
$\rho(\mathcal{L}(\Gamma))=\Theta$, so $||\Gamma||_{\delta,\nu+1}\le||\Theta||_{\delta,\nu}$, and so $\Gamma_{ij}{}^j=0$.
\end{lemma}

\noindent{\it Proof.} We have assumed throughout that $m\ge3$. For each pair of indices $\{i,j\}$, not necessarily distinct,
choose $k_{ij}=k_{ji}$ distinct from $i$ and from $j$. Define the indefinite integral
$$\textstyle\left(\int_k\Theta\right)(z):=z_k\int_0^1\Theta(z_1,...,z_{k-1},tz_k,z_{k+1},...,z_m)dt\,.$$
Let $\textstyle\Gamma_{ij}{}^\ell:=\delta_{k_{ij}}^\ell\int_{k_{ij}}\Theta_{ij}$.
It is immediate that $||\Gamma||_{\delta,\nu+1}\le||\Theta||_{\delta,\nu}$. Since $k_{ij}$ is distinct from $i$ and $j$, we have
that $\Gamma_{ij}{}^j=0$. Furthermore, $\Gamma(x)$ is real if $x$ is real. Finally, we use Equation (\ref{eqn-4.a}) to complete the proof
by checking:
$$(\rho(\mathcal{L}(\Gamma)))_{jk}=\partial_{x_i}\Gamma_{jk}{}^i=\Theta_{jk}\,.\qquad\qquad\qedbox$$

\begin{proof}[Proof of Lemma \ref{lem-4.2}] Let $\mathcal{A}\in\mathfrak{A}(V)$. The Christoffel symbols $\Gamma_1$ are as in Lemma
\ref{lem-2.1}. Since $\Gamma_1$ is a homogeneous linear polynomial, there is a constant $C_1>0$ so 
$$||\Gamma_1||_{\delta,1}\le C_1\quad\text{and}\quad||\rho(\mathcal{R}^{\nabla_1})-\rho(\mathcal{A})||_{\delta,2}=
||\rho(\Gamma_1\star\Gamma_1)||_{\delta,2}\le C_1^2$$ 
for any $\delta>0$. Let $C:=8m^2C_1>0$.
Choose $\delta>0$ so that
\begin{equation}\label{eqn-4.b}
C\delta^2<1\quad\text{and}\quad \frac{C^3\delta^2}{1-C^2\delta^2}\le C_1\,.
\end{equation}

If $\nu=1$, Assertion (3) of Lemma \ref{lem-4.2} holds vacuously and Assertions (2) and (4) hold since $C>C_1$.
Thus we may proceed by induction to establish Assertions (2), (3), and (4). We assume
$\Gamma_1$,...,$\Gamma_\nu$ have been chosen with the desired properties. Use Lemma \ref{lem-4.3} to choose $\Gamma_{\nu+1}\in
\mathcal{H}(\delta,2\nu+1,S^2(V^*)\otimes V)$ so that 
$$
\rho_s(\mathcal{L}(\Gamma_{\nu+1}))=-\rho_s(\mathcal{R}^{\nabla_{\nu}})+\rho_s(\mathcal{A})\quad\text{and}\quad
||\Gamma_{\nu+1}||_{\delta,2\nu+1}\le C^{2\nu+1}\,.
$$
We have the estimate:
\begin{equation}\label{eqn-4.d}
||\rho_s(\Gamma_\mu\star\Gamma_{\nu+1})||_{\delta,2\mu+2\nu}\le4m^2||\Gamma_\mu||_{\delta,2\mu-1}\cdot||\Gamma_\nu||_{\delta,2\nu+1}\,.
\end{equation}
As $\rho_s(\mathcal{R}^{\nabla_\nu}+\mathcal{L}(\Gamma_{\nu+1}))(z)=\rho_s(\mathcal{A})$ for any $z\in B_\delta$,  Equations
(\ref{eqn-4.a}) and (\ref{eqn-4.d}) yield
\begin{eqnarray*}
&&|\{\rho_s(\mathcal{R}^{\nabla_{\nu+1}})(z)-\rho_s(\mathcal{A})\}_{ij}|\\
&=&\textstyle|\{\rho_s[\mathcal{R}^{\nabla_\nu}+\mathcal{L}(\Gamma_{\nu+1})
+\sum_{\mu\le\nu}\Gamma_\nu\star\Gamma_{\nu+1}
+\frac12\Gamma_{\nu+1}\Gamma_{\nu+1}](z)-\rho_s(\mathcal{A})\}_{ij}|\\
&=&\textstyle|\{\rho_s[\sum_{\mu\le\nu}\Gamma_\nu\star\Gamma_{\nu+1}+\frac12\Gamma_{\nu+1}\Gamma_{\nu+1}](z)]\}_{ij}|\\
&\le&4m^2\{C_1|z|+C^3|z|^3+...+C^{2\nu+1}|z|^{2\nu+1}\}C^{2\nu+1}|z|^{2\nu+1}\\
&\le&4m^2\{C_1+C^3\delta^2+...+C^{2\nu+1}\delta^{2\nu}\}C^{2\nu+1}|z|^{2\nu+2}\,.
\end{eqnarray*}
Estimating using a geometric series and applying Equation (\ref{eqn-4.b}) completes the inductive step by showing
\begin{eqnarray*}
&&|\{\rho_s(\mathcal{R}^{\nabla_{\nu+1}})(z)-\rho_s(\mathcal{A})\}_{ij}|
\le 4m^2\left\{C_1+\frac{C^3\delta^2}{1-C^2\delta^2}\right\}C^{2\nu+1}|z|^{2\nu+2}\\
&&\qquad\le 8m^2C_1C^{2\nu+1}|z|^{2\nu+2}\le C^{2\nu+2}|z|^{2\nu+2}\,.
\end{eqnarray*}
The proof of Lemma \ref{lem-4.2} and thereby of Theorem \ref{thm-4.1} and thus of Theorem \ref{thm-1.4} is now complete.
\end{proof}

\section{The proof of Theorem \ref{thm-1.7}}\label{sect-5}

\begin{proof} 
Let $\{e_i\}$ be a basis for $V$ and let $\{x_i\}$ be the associated coordinate system on $V$. Let $\theta$ be a $1$-form. Motivated by
Theorem \ref{thm-1.6}, we define a connection $\nabla^\theta$ so
$$\nabla^\theta_xy=\theta(x)y+\theta(y)x$$
if $x$ and $y$ are coordinate vector fields. Set 
$\Psi(x,y)=x\theta(y)$ and let $H$ be as in Equation (\ref{eqn-1.d}). Then:
\begin{eqnarray*}
\mathcal{R}^{\nabla^\theta}(x,y)z&=&\theta(x)\theta(y)z+\theta(x)\theta(z)y
-\theta(y)\theta(x)z-\theta(y)\theta(z)x\\
&+&x(\theta(y))z+x(\theta(z))y-y(\theta(x))z-y(\theta(z))x\\
&=&H(\theta\otimes\theta+\Psi)\,.
\end{eqnarray*}
Consequently $\nabla^\theta$ is projectively flat. Let $\Theta\in V^*\otimes V^*$. Set $\theta=x_i\Theta_{ij}dx_j$. Then $\theta(0)=0$ and
$\Psi(0)=\Theta$ so
$$\rho(\mathcal{R}^{\nabla^\theta})(0)=(1-m)\Theta_s-(m+1)\Theta_a\,.$$
Consequently, given any $\mathcal{A}\in\frak{A}(V)$ with
$\mathcal{P}(\mathcal{A})=0$ there exists a torsion free projectively flat connection $\nabla^\theta$ so that
$\mathcal{R}^{\nabla^\theta}_0=\mathcal{A}$. This proves Theorem  \ref{thm-1.7} (1). Furthermore, if $\Theta$ is symmetric, then
$d\theta=0$. Thus $\Psi$ is symmetric for any point
$P\in V$ and $\mathcal{R}^{\nabla^\theta}$ is Ricci symmetric. This proves Theorem \ref{thm-1.7} (2).\end{proof}

\section{The proof of Theorem \ref{thm-1.8}}\label{sect-6}

\begin{proof} Let
$\mathcal{R}^\nabla(x,y;z)=(\nabla_z\mathcal{R}^\nabla)(x,y)$  be the covariant derivative of the curvature; we then have the {\it
second Bianchi identity}:
\begin{equation}\label{eqn-6.a}
0=\mathcal{R}^\nabla(x,y;z)+\mathcal{R}^\nabla(y,z;x)+\mathcal{R}^\nabla(z,x;y)\,.
\end{equation}

Suppose $\rho(\mathcal{R}^\nabla)\in\Lambda^2(V^*)$. Let $\omega_{ij}:=-\frac1{m+1}\rho(\partial_{x_i},\partial_{x_j})$.
By Equations (\ref{eqn-1.d}) and (\ref{eqn-x}),
$$\mathcal{R}^\nabla(\partial_{x_i},\partial_{x_j})\partial_{x_k}=2\omega_{ij}\partial_{x_k}+\omega_{ik}\partial_{x_j}-\omega_{jk}\partial_{x_i}\,.$$
 Covariantly differentiating this relation yields:
\begin{eqnarray*}
&&\mathcal{R}^\nabla(\partial_{x_i},\partial_{x_j};\partial_{x_\ell})\partial_{x_k}
   =2\omega_{ij;\ell}\partial_{x_k}+\omega_{ik;\ell}\partial_{x_j}-\omega_{jk;\ell}\partial_{x_i},\\
&&\mathcal{R}^\nabla(\partial_{x_j},\partial_{x_\ell};\partial_{x_i})\partial_{x_k}
   =2\omega_{j\ell;i}\partial_{x_k}+\omega_{jk;i}\partial_{x_\ell}-\omega_{\ell k;i}\partial_{x_j},\\ 
&&\mathcal{R}^\nabla(\partial_{x_\ell},\partial_{x_i};\partial_{x_j})\partial_{x_k}
   =2\omega_{\ell i;j}\partial_{x_k}+\omega_{\ell k;j}\partial_{x_i}-\omega_{i
k;j}\partial_{x_\ell}\,.
\end{eqnarray*}
Summing and applying the second Bianchi identity of Equation (\ref{eqn-6.a}) yields
\begin{eqnarray}
0&=&(2\omega_{ij;\ell}+2\omega_{j\ell;i}+2\omega_{\ell i;j})\partial_{x_k}\label{eqn-6.b}\\
&+&(\omega_{ik;\ell}-\omega_{\ell k;i})\partial_{x_j}
+(\omega_{\ell k;j}-\omega_{jk;\ell})\partial_{x_i}+(\omega_{jk;i}-\omega_{i k;j})\partial_{x_\ell}\,.\nonumber
\end{eqnarray}
Let $\{i,j,\ell\}$ be distinct indices. Set $k=i$. Examining the coefficient of $\partial_{x_j}$ in Equation (\ref{eqn-6.b}) yields
$$0=\omega_{ii;\ell}-\omega_{\ell i;i}=\omega_{i\ell;i}\,.$$
Polarizing this identity then
yields
$$\omega_{i\ell;j}+\omega_{j\ell;i}=0\quad\text{and}\quad\omega_{\ell i;j}+\omega_{\ell j;i}=0\,.$$
Next we set $k=\ell$ and examine the coefficient of $\partial_{x_k}$ in Equation  (\ref{eqn-6.b}) to see
\begin{eqnarray*}
0&=&2\omega_{ij;k}+2\omega_{jk;i}+2\omega_{ki;j}+\omega_{jk;i}-\omega_{i k;j}\\
&=&2\omega_{ij;k}+3\omega_{jk;i}+3\omega_{ki;j}
=-2\omega_{kj;i}+3\omega_{jk;i}-3\omega_{kj;i}=8\omega_{jk;i}\,.
\end{eqnarray*}
Thus if $\{x,y,z\}$ are linearly independent vectors, then $\nabla_x\omega(y,z)=0$; since the set of all triples of linearly independent
vectors is dense in the set of all triples, this relation holds by continuity for all
$\{x,y,z\}$ and thus
$\nabla\omega=0$. We compute:
\begin{eqnarray*}
0&=&\{(\nabla_x\nabla_y-\nabla_y\nabla_x-\nabla_{[x,y]})\omega\}(z,w)\\
&=&\omega(\mathcal{R}^\nabla(x,y)z,w)+\omega(z,\mathcal{R}^\nabla(x,y)w)\\
&=&4\omega(x,y)\omega(z,w)+2\omega(x,z)\omega(y,w)-2\omega(x,w)\omega(y,z)\,.
\end{eqnarray*}
Set $x=z$ and $y=w$ to see that $6\omega(x,y)^2=0$. Consequently $\omega=0$ so
$\mathcal{R}=0$.
\end{proof}

\section*{Acknowledgments} Research of P. Gilkey partially supported by the
Max Planck Institute in the Mathematical Sciences (Leipzig), by Project MTM2006-01432 (Spain), and by the Complutense Universidad de Madrid
(SEJ2007-67810). Research of S. Nik\v cevi\'c partially supported by Project 144032 (Srbija). Research of D. Westerman
partially supported by the University of Oregon.

\end{document}